\documentclass[11pt]{article}

\usepackage[utf8]{inputenc}
\usepackage[T1]{fontenc}
\usepackage{amsmath, amssymb, amsthm, amsfonts}
\usepackage{url}
\usepackage{epsfig}
\usepackage{microtype}

\newtheorem{theorem}{Theorem}
\newtheorem{lemma}{Lemma}
\newtheorem{proposition}{Proposition}
\newtheorem{corollary}{Corollary}

\theoremstyle{definition}

\newtheorem{remark}{Remark}
\newtheorem{example}{Example}

\begin{document}

\begin{center}
{\LARGE \bf A formula of counting divisors in integers rings: a generalization of the divisor function $d_0(n)$} \\
\vskip 20pt
{\bf \'Angel Mart\'inez-Avelar}\\
{\it Departamento de Matem\'aticas, UAM-I, M\'exico}\\
{\tt avelar@xanum.uam.mx}\\
\vskip 10pt
{\bf Mario Pineda-Ruelas}\\
{\it Departamento de matem\'aticas, UAM-I, M\'exico}\\
{\tt mpr@xanum.uam.mx}\\
\end{center}

\vskip 30pt

\begin{abstract}
\noindent In this paper we establish a formal connection between the structure of ideals in integers rings and the theory of additive combinatorics. For integers rings with cyclic class groups, we prove a structural theorem demonstrating that every non-zero ideal can be decomposed into a maximal principal part and a product of ideals whose total length is bounded by the Davenport constant. With this decomposition we find divisors for generators of the ideal $I=(\alpha, \beta)$. The central result of this work is the derivation of a closed formula using character theory over finite abelian groups to count the exact number of zero-sum subsequences of a given sequence. Under the established correspondence between principal ideals and zero-sum sequences, this formula provides a precise counting of the principal ideal divisors of any given ideal, and therefore counting common divisors of generators of the ideal $I=(\alpha,\beta)$. This result constitutes a natural generalization of the classical divisor function $d_0(n)$ from unique factorization domains to any Dedekind domain with a finite class group. Finally, we characterize irreducible elements in $\mathcal{O}_K$ based on the counting of these zero-sum subsequences.
\end{abstract}

\section{Preliminaries}

Let $K$ be a number field and $\mathcal{O}_K$ its integers ring with ideal class group $G = Cl(K)$ of order $h$.  If $I$ is an ideal of $\mathcal{O}_K$, we write $[I]$ to represent the class of $I$ in $Cl(K)$.  
In $\mathcal{O}_K$, any ideal $I\not= 0$ factors uniquely as a product of prime ideals $I = \mathfrak{p}_1^{e_1} \cdots \mathfrak{p}_g^{e_g}$ (for a detailed development of ideal factorization, see \cite{Conrad}), and its principality depends exclusively on the image of these factors in the class group. To approach the problem of counting principal ideal divisors from an analytical perspective, it is necessary to specify the algebraic framework and the group theory tools that we will use.

\subsection{The Class Group and Zero-Sum Sequences}

The study of zero-sum sequences has been a central topic in additive combinatorics (see the survey by Gao and Geroldinger \cite{GG06})

A \textbf{zero-sum sequence} in a finite abelian  group $G$ is an element $S = (g_1, g_2, \dots, g_n)$, with $g_i \in G$, such that:
\begin{equation}
    \sum_{i=1}^{n} g_i = 0_G
\end{equation}
where $0_G$ is the identity element of $G$ in additive notation.

In the context of $\mathcal{O}_K$, if an ideal $(\alpha)$ has a factorization $(\alpha) = \mathfrak{p}_1^{e_1} \dots \mathfrak{p}_g^{e_g}$, the sequence formed by the classes of its prime factors $[\mathfrak{p}_i]$ (repeated according to their multiplicity $e_i$) is, by definition, a zero-sum sequence in $Cl(K)$.

Let $R$ be a Dedekind domain and $I \subseteq R$ be a non-zero ideal with prime ideal factorization $I = \prod_{i=1}^{g} \mathfrak{p}_i^{e_i}$. An \textbf{ideal divisor} of $I$ is any ideal $J$ of the form $J = \prod_{i=1}^{g} \mathfrak{p}_i^{f_i}$, with $0 \le f_i \le e_i$. We say that:
\begin{enumerate}
    \item $J$ is a \textbf{principal ideal divisor} if $J = (a)$ for some $a \in R$.
    \item $(b)$ is a \textbf{maximal principal ideal divisor} of $I$ if it is principal and there is no principal ideal divisor $J$ of $I$ such that $(b) \subsetneq J \subseteq I$.
    \item $(\ell)$ is a \textbf{common principal ideal divisor} of $(\alpha)$ and $(\beta)$ if $(\alpha)\subseteq(\ell)$ and $(\beta)\subseteq(\ell)$.
    \item $(\ell)$ is a \textbf{maximal common principal ideal divisor} of $(\alpha)$ and $(\beta)$ if for any other common principal ideal divisor $(\gamma)$ it holds that $(\gamma) \mid (\ell)$.
\end{enumerate}

\subsection{Character Theory and Davenport Constant}

The analytical study of zero-sum sequences is facilitated by using the dual group of $G$, denoted by $\hat{G}$:

\[ \hat{G}:=\{ \chi:G \to \mathbb{C}^{*}\mid \chi \ \text{is a group homomorphism}\}. \]

Since $G$ is a finite abelian group, it is well known that $G \cong \hat{G}$, implying $|\hat{G}| = |G| = h$. The fundamental tool for detecting elements that satisfy the zero-sum condition is the \textbf{orthogonality relation} of the characters, which allows us to "filter" for the identity element within a sum:

\begin{lemma}
Let $G$ be a finite abelian group and $x \in G$. Then:
\begin{equation}
    \frac{1}{h} \sum_{\chi \in \hat{G}} \chi(x) = 
    \begin{cases} 
    1 & \text{if } x = 0_G \\
    0 & \text{if } x \neq 0_G 
    \end{cases}
\end{equation}
where $h = |G|$.
\end{lemma}
\begin{proof}
    See Theorem 6.13 of \cite{Ap}.
\end{proof}

\medskip
Using these character-theoretic tools, one can study the combinatorial invariants of $G$. A central object in this study is the \textbf{Davenport constant}. Using these character-theoretic tools, one can study the combinatorial invariants of $G$. A central object in this study is the \textbf{Davenport constant}. The study of this constant was proposed by Professor Harold Davenport in 1966 in connection with the study of factorization problems in integer rings in number fields. This leads to the involvement of zero-sum problems \cite{olson}.

Let $G$ be an additive abelian group. The Davenport constant $\mathcal{D}(G)$ is defined as the smallest integer $k \in \mathbb{N}$ such that every sequence $S$ over $G$ of length $|S| \ge k$ has a non-empty zero-sum subsequence. In the specific case of cyclic groups, the Davenport constant coincides exactly with the order of the group, a result that simplifies many calculations in the context of class groups (see Lemma 1.4.9 of \cite{GH06}).

\medskip
While the Davenport constant provides a combinatorial upper bound for factorizations, its relevance stems from the arithmetic of the integers ring $\mathcal{O}_K$. Specifically, the failure of an ideal to be principal is what necessitates these combinatorial methods. The following proposition provides a practical criterion for identifying such non-principal ideals by examining their norms.

\begin{proposition}\label{princi}
Let $\mathcal{O}_K$ be the integers ring of a number field $K$, and let $\mathfrak{p} \subset \mathcal{O}_K$ be a prime ideal. If there exists no element $\alpha \in \mathcal{O}_K$ such that $|N_{K/\mathbb{Q}}(\alpha)| = N(\mathfrak{p})$, then $\mathfrak{p}$ is not a principal ideal.
\end{proposition}

\begin{proof}
Suppose that $\mathfrak{p}$ is a principal ideal. Then there exists some element $\alpha \in \mathcal{O}_K$ such that $\mathfrak{p} = (\alpha)$. By the properties of the numerical norm of an ideal, we have:
\[
N(\mathfrak{p}) = N((\alpha)) = |N_{K/\mathbb{Q}}(\alpha)|
\]
However, this contradicts the hypothesis that no such $\alpha$ exists in $\mathcal{O}_K$. Therefore, $\mathfrak{p}$ must be non-principal.
\end{proof}

\section{Structure of Ideals}

We now establish the structural result that motivates the use of combinatorial tools. The following theorem demonstrates that, if $Cl(K)$ is a cyclic group, every ideal can be reduced to a principal part and a residue whose length is strictly bounded by the Davenport constant of the class group, ensuring that said residue contains no proper principal divisors.

\begin{theorem}\label{teoprinc}
Let $K$ be a number field, $\mathcal{O}_K$ its integers ring and $I \subseteq \mathcal{O}_K$ an ideal. Suppose that $Cl(K)$ is cyclic of order $h$. Then there exist $\gamma \in \mathcal{O}_K$, distinct ideals $T_1, \dots, T_{m}$ and $c_1,\ldots,c_m\in\mathbb{N}$ such that $c_1+\ldots+c_m\leq h-1$ satisfying
\[ I = (\gamma) T_1^{c_1} \cdots T_{m}^{c_m}, \]
where for each $i$, $T_i$ is a non-principal prime ideal or $T_i=\mathcal{O}_K$.
\end{theorem}

\begin{proof}
   If $I=\mathcal{O}_K$ or $I=\{0\}$, then the result is trivial. Suppose that $Cl(K) \cong \mathbb{Z}/h\mathbb{Z}$, and $I=(\alpha,\beta)$ is a non-zero proper ideal of $\mathcal{O}_K$. Consider the following prime ideal factorizations:
\[ (\alpha) = P_1^{a_1} \cdots P_g^{a_g}, \]
\[ (\beta) = Q_1^{b_1} \cdots Q_r^{b_r}. \]
Suppose that $(\alpha)$ and $(\beta)$ have no common principal ideal divisor other than the trivial one. Since $I\neq \mathcal{O}_K$, there exist $1\leq i\leq g$ and $1\leq j\leq r$ such that $P_i=Q_j$. Without loss of generality, assume such equalities occur in the first $m$ factors, i.e., $P_i=Q_i$ for $i=1,\ldots, m$. For each $1\leq i\leq m$, we let $T_i=P_i=Q_i$ and $c_i=\min\{a_i,b_i\}$. Then

\begin{equation*}
\begin{split}
I &= (\alpha,\beta) = (\alpha) + (\beta) \\
  &= T_1^{c_1} \cdots T_m^{c_m} (P_1^{\alpha_1} \cdots P_m^{\alpha_m} P_{m+1}^{a_{m+1}} \cdots P_g^{a_g} + Q_1^{\beta_1} \cdots Q_m^{\beta_m} Q_{m+1}^{b_{m+1}} \cdots Q_r^{b_r}),
\end{split}
\end{equation*}
where each $T_i$ is non-principal, $T_i\neq T_j$ if $i\neq j$ and $c_i+\alpha_i=a_i$, $c_i+\beta_i=b_i$ for $i=1,\ldots, m$. Since there are no more common factors in the sum, we have

\[ I=(1)T_1^{c_1}\cdots T_m^{c_m}. \] 
If $c_1+\ldots+c_m\geq h$, then $[T_1]^{c_1}\cdots [T_m]^{c_m}$ is a sequence of $h$ or more elements in $Cl(K)$. Given that $\mathcal{D}(Cl(K))=h$, there exist $i_1,\ldots,i_s$ such that $[T_{i_1}]\cdots[T_{i_s}]=1$, which is not possible because there are no non-trivial principal ideal divisors. Therefore $c_1+\ldots+c_m\leq h-1$. Suppose there exists a non-trivial common principal ideal divisor $(\gamma)$ of $(\alpha)$ and $(\beta)$, and we choose $(\gamma)$ to be maximal. Then 

\[ I=(\alpha,\beta)=(\gamma)(a,b), \]
where $(\gamma a)=(\alpha)$, $(\gamma b)=(\beta)$ and the ideals $(a)$, $(b)$ have no common principal ideal divisors by the maximality of $(\gamma)$. From the previous case, we have 

\[ (a,b)=(1)T_1^{c_1}\cdots T_m^{c_m}, \]
where $c_1+\ldots+c_m\leq h-1$ and $T_i$ is a non-principal prime ideal for all $i$. Therefore 

\[ I=(\gamma)T_1^{c_1}\cdots T_m^{c_m}. \]
\end{proof}

\begin{remark}
    This result is true for Dedekind domains with finite cyclic class group. 
\end{remark}

\begin{remark}
The proof of Theorem \ref{teoprinc} clarifies that principality in Dedekind domains with a finite cyclic class group is a finite and bounded phenomenon. Once the maximal principal divisor $(\gamma)$ is extracted, counting the principal ideal divisors of the ideal $T_1^{c_1} \cdots T_m^{c_m}$ is trivial (only the trivial one and itself exist if the total sum were zero). Furthermore, the maximal principal divisor is not unique.
\end{remark}

\begin{remark}
    If $I=(\alpha,\beta)=(\gamma)T_1^{c_1}\cdots T_m^{c_m}$, then $\gamma\mid \alpha$, $\gamma\mid \beta$. Furthermore, if $\gamma\in I$, then $I=(\gamma)$.
\end{remark}

\begin{example}
We take the field $K=\mathbb{Q}(\sqrt{10})$ and consider the ideal 

\[ I=(18,10+\sqrt{10})=(2,\sqrt{10})(3,1+\sqrt{10})^2, \]
we have two maximal principal ideal divisors: 

\[ (-2+\sqrt{10})=(2,\sqrt{10})(3,1+\sqrt{10}) \quad \text{and} \quad (1+\sqrt{10})=(3,1+\sqrt{10})^2. \]    
\end{example}
\noindent Moreover, maximal principal ideal divisors do not necessarily have the same number of prime factors.

\begin{example}
    Consider the real quadratic field $K = \mathbb{Q}(\sqrt{195})$. The discriminant is $780 =2^2\cdot 3 \cdot 5 \cdot 13$, and $\mathcal{O}_K = \mathbb{Z}[\sqrt{195}]$. The class group of $K$ is isomorphic to the Klein group $C_2 \times C_2$.
Consider the following prime ideals:
\[
\mathfrak{p}_3 = (3, \sqrt{195}), \quad \mathfrak{p}_5 = (5, \sqrt{195}), \quad \mathfrak{p}_{13} = (13, \sqrt{195}).
\]
In $Cl(K)$, these ideals satisfy the following relations:
\begin{enumerate}
    \item $[\mathfrak{p}_i] \neq 1$ for $i=3,5,13$.
    \item $[\mathfrak{p}_i]^2 = 1$ for $i=3,5,13$.
    \item $[\mathfrak{p}_3][\mathfrak{p}_5][\mathfrak{p}_{13}] = 1$, which implies that the product $\mathfrak{p}_3 \mathfrak{p}_5 \mathfrak{p}_{13} = (\sqrt{195})$ is principal.
    \item If $i\neq j$, then $[\mathfrak{p}_i][\mathfrak{p}_j]\neq 1$.
\end{enumerate}

Let $I = \mathfrak{p}_3^2 \cdot \mathfrak{p}_5^2 \cdot \mathfrak{p}_{13}$.
We will construct two maximal principal ideal divisors of $I$ with different lengths. Let $J_1=\mathfrak{p}_3^2\cdot\mathfrak{p}_5^2=(3)(5)=(15)$ and $J_2=\mathfrak{p}_3\mathfrak{p}_5\mathfrak{p}_{13}=(\sqrt{195})$. Since $\mathfrak{p}_{13}$ is not principal, $J_1$ is a maximal principal ideal divisor of $I$. Using the previous argument, $J_2$ is also a maximal principal ideal divisor of $I$. 
\end{example}

\begin{example}\label{ej1}
    
According to the LMFDB database \cite{LMFDB}, the ideal class group of the field $K=\mathbb{Q}(\sqrt{219})$ is cyclic of order $4$. Since $N(3+\sqrt{219})=2\cdot 3\cdot 5\cdot 7$, using Dedekind's Theorem, the ideal $(3+\sqrt{219})$ has the factorization 

\[ (3+\sqrt{219})=(2, 1+\sqrt{219}) (3, \sqrt{219})(5, 3+\sqrt{219})(7, 3+\sqrt{219}), \]
where each factor is a non-principal prime ideal by Proposition \ref{princi}. We will compute all possible principal ideal divisors of $(3+\sqrt{219})$. Each factor provides a non-principal ideal divisor. The ideal divisors with two factors are:
\begin{itemize}
    \item $(2,1+\sqrt{219})(3,\sqrt{219})=(-15+\sqrt{219})$,

    \item $(2,1+\sqrt{219})(5,3+\sqrt{219})=(10,3+\sqrt{219})$,

    \item $(2,1+\sqrt{219})(7,3+\sqrt{219})=(14,3+\sqrt{219})$,

    \item $(3,\sqrt{219})(5,3+\sqrt{219})=(15,3+\sqrt{219})$,

    \item $(3,\sqrt{219})(7,3+\sqrt{219})=(21,3+\sqrt{219})$,

    \item $(5,3+\sqrt{219})(7,3+\sqrt{219})=(-29+2\sqrt{219})$.
\end{itemize}

The ideal divisors with three factors are:

\begin{itemize}
    \item $(2,1+\sqrt{219})(3,\sqrt{219})(5,3+\sqrt{219})=(30,3+\sqrt{219})$,

    \item $(2,1+\sqrt{219})(3,\sqrt{219})(7,3+\sqrt{219})=(42,3+\sqrt{219})$,

    \item $(2,1+\sqrt{219})(5,3+\sqrt{219})(7,3+\sqrt{219})=(70,3+\sqrt{219})$,

    \item $(3,\sqrt{219})(5,3+\sqrt{219})(7,3+\sqrt{219})=(105,3+\sqrt{219})$.
\end{itemize}

There is only one ideal divisor with four and zero factors, which produce trivial principal ideal divisors. Therefore, the principal ideal divisors are:

\begin{itemize}

    \item $\mathcal{O}_K=(1)$,

    \item $(2,1+\sqrt{219})(3,\sqrt{219})=(-15+\sqrt{219})$,

    \item $(5,3+\sqrt{219})(7,3+\sqrt{219})=(-29+2\sqrt{219})$,

    \item $(2,1+\sqrt{219})(3,\sqrt{219})(5,3+\sqrt{219})(7,3+\sqrt{219})=(3+\sqrt{219})$.
    
\end{itemize}

From the following equality 

\[ 3+\sqrt{219}=(-15+\sqrt{219})(-44-3\sqrt{219})=(-29+2\sqrt{219})(15+\sqrt{219}). \]

\noindent The principal ideal divisors have produced divisors of $3+\sqrt{219}$. 

\end{example}
This motivates the need for a general formula to count how many divisors an element has, such as the one we will present in Section 3, for cases where the ideals have not been reduced to this canonical form.

\section{Main Result: Counting Zero-Sum Subsequences}

In this section, we find a formula to count the number of zero-sum subsequences of a given sequence, which, under the established correspondence, is equivalent to counting the number of principal ideal divisors of an element $\alpha \in \mathcal{O}_K$.

\begin{theorem}\label{teo2}[Subsequence Counting Formula]
Let $G$ be a finite abelian group of order $h$ and let $S = (r_1^{e_1}, r_2^{e_2}, \dots, r_g^{e_g})$ be a sequence in $G$, where each $r_i$ is an element of $G$ with multiplicity $e_i$. The number $n(S)$ of zero-sum subsequences of $S$ is given by:
\begin{equation}
    n(S) = \frac{1}{h} \sum_{\chi \in \hat{G}} \prod_{i=1}^{g} \left( \sum_{k=0}^{e_i} \chi(r_i)^k \right)
\end{equation}
where $\hat{G}$ is the dual group of characters of $G$.
\end{theorem}

\begin{proof}
Consider the set of all possible subsequences of $S$, which can be represented by tuples of exponents $(f_1, f_2, \dots, f_g)$ such that $0 \le f_i \le e_i$. A subsequence is zero-sum if and only if:
\begin{equation}
    \sum_{i=1}^{g} f_i r_i = 0_G
\end{equation}
We define $n(S)$ as the sum over all possible tuples, using the indicator function of the character orthogonality relation:
\begin{equation}
    n(S) = \sum_{f_1=0}^{e_1} \dots \sum_{f_g=0}^{e_g} \left( \frac{1}{h} \sum_{\chi \in \hat{G}} \chi \left( \sum_{i=1}^{g} f_i r_i \right) \right)
\end{equation}
By the multiplicative property of characters ($\chi(a+b) = \chi(a)\chi(b)$) and rearranging the summations:
\begin{equation}
    n(S) = \frac{1}{h} \sum_{\chi \in \hat{G}} \sum_{f_1=0}^{e_1} \dots \sum_{f_g=0}^{e_g} \prod_{i=1}^{g} \chi(r_i)^{f_i}
\end{equation}
Factoring the product of the independent sums for each component $i$:
\begin{equation}
    n(S) = \frac{1}{h} \sum_{\chi \in \hat{G}} \prod_{i=1}^{g} \left( \sum_{f_i=0}^{e_i} \chi(r_i)^{f_i} \right)
\end{equation}
Each term within the product is a finite geometric series in $\mathbb{C}$. This completes the proof.
\end{proof}

\begin{corollary}
Let $\alpha \in \mathcal{O}_K$ and $(\alpha) = \mathfrak{p}_1^{e_1} \cdots \mathfrak{p}_g^{e_g}$ be the prime ideal factorization. The number of principal ideal divisors of $(\alpha)$, denoted by $n(\alpha)$, is given by
\begin{equation}
    n(\alpha) = \frac{1}{h} \sum_{\chi \in \hat{G}} \prod_{i=1}^{g}\left( \sum_{k=0}^{e_i} \chi([\mathfrak{p}_i])^k \right)
\end{equation}
where $h=\mid Cl(K)\mid$.
\end{corollary}
Theorem \ref{teo2} higlights by its appearence in classical results in elementary number theory. The counting formula $n(\alpha)$ is a natural generalization of the classical divisor function $d_0(n)$ from elementary number theory.

\begin{corollary}[Generalization of the Classical Divisor Function]
If $\mathcal{O}_K$ is a Principal Ideal Domain, the formula $n(\alpha)$ reduces to the standard counting of divisors in a Unique Factorization Domain.
\end{corollary}

\begin{proof}
If $\mathcal{O}_K$ is a PID, then $\hat{G}=\{\chi_0\}$, where $\chi_0(g) = 1$ for all $g \in G$. Substituting these values into the main formula, we have:
\begin{equation*}
    n(\alpha) =  \sum_{\chi \in \{\chi_0\}} \prod_{i=1}^{g} \left( \sum_{k=0}^{e_i} \chi_0([\mathfrak{p}_i])^k \right) = \prod_{i=1}^{g} \left( \sum_{k=0}^{e_i} 1 \right).
\end{equation*}
Since each inner sum contains exactly $e_i + 1$ terms, the expression simplifies to:
\begin{equation*}
    n(\alpha) = \prod_{i=1}^{g} (e_i + 1).
\end{equation*}

\end{proof}

This identity is precisely the well-known formula for the number of divisors of an element in a UFD based on its prime factorization. This consistency confirms that our character-theoretic approach properly accounts for the lack of unique factorization by means of the orthogonality relations, recovering the classical case when the class group is trivial.

As an immediate consequence of the Subsequence Counting Theorem, we recover the fundamental result for the arithmetic of Dedekind domains. The structure of the class group dictates divisibility in the integers ring through the following relationship:

\begin{lemma}
    Let $\alpha\in\mathcal{O}_K$ such that $n(\alpha)=2$. Then $\alpha$ is irreducible. 
\end{lemma}

\begin{example}
Let $\alpha \in \mathcal{O}_K$  such that its associated principal ideal factors as $(\alpha) = \mathfrak{p}_1 \mathfrak{p}_2$ or $(\alpha) = \mathfrak{p}_1 \mathfrak{p}_2 \mathfrak{p}_3$, where each $\mathfrak{p}_i$ is a non-principal prime ideal. Under these conditions, $\alpha$ is an irreducible element in $\mathcal{O}_K$. 

\begin{itemize}
    \item \textbf{Case 1:} Irreducibility is immediate since no prime factor is principal by hypothesis.
    
    \item \textbf{Case 2:} It suffices to note that any product of the form $\mathfrak{p}_i \mathfrak{p}_j$ must necessarily be a non-principal ideal; otherwise, if $\mathfrak{p}_i \mathfrak{p}_j = (\gamma)$, then the ideal $(\alpha) = (\gamma) \mathfrak{p}_k$ would be non-principal for being the product of a principal ideal and a non-principal one, which contradicts the principality of $(\alpha)$. 
\end{itemize}

Therefore, the sequence of classes of $(\alpha)$ does not possess proper zero-sum subsequences, guaranteeing that $\alpha$ does not admit a factorization into non-trivial elements.
\end{example}

The above is not true for four or more factors; for this, consider the field $K=\mathbb{Q}(\sqrt{399})$ and the non-principal prime ideals $\mathfrak{p}_1=(2,1+\sqrt{399})$, $\mathfrak{p}_2=(3,\sqrt{399})$, $\mathfrak{p}_3=(7,\sqrt{399})$ and $\mathfrak{p}_4=(19,\sqrt{399})$. We have that 

\[ (\sqrt{399})=\mathfrak{p}_1\mathfrak{p}_2\mathfrak{p}_3\mathfrak{p}_4, \]

but 

\[ (19+\sqrt{399})=\mathfrak{p}_1\mathfrak{p}_2 \quad \text{and} \quad (2+\sqrt{399})=\mathfrak{p}_3\mathfrak{p}_4.  \]

\vskip20pt\noindent {\bf Acknowledgements.} We thank the SECIHTI   for the valuable support to first author  Ángel Martínez-Avelar  CVU 1036860. The authors also acknowledge the use of SageMath \cite{Sage} for computational verification of the examples presented in this work.

\end{document}